\documentclass[final]{siamltex}
\usepackage{amsmath}
\usepackage{amsfonts}
\usepackage{graphicx}
\usepackage{amscd}
\newcommand*{\I}{{\mathsf{id}}}
\newcommand*{\G}{{\mathcal{G}}}

\newcommand*{\g}{{\mathfrak{g}}}
\newcommand*{\M}{{\mathcal{M}}}
\newcommand*{\R}{{\mathbb{R}}}
\newcommand*{\X}{{\mathfrak{X}}}
\newcommand*{\oo}{{\mathsf{o}}}
\newcommand*{\Ad}{{\mathrm{Ad}}}
\newcommand*{\ra}{{\rightarrow}}

\title{Stochastic Lie group integrators}
\author{Simon J.A. Malham\thanks{Maxwell Institute 
for Mathematical Sciences and School of Mathematical and Computer Sciences, 
Heriot-Watt University, Edinburgh EH14 4AS, UK. (\texttt{S.J.Malham@ma.hw.ac.uk},
\texttt{A.Wiese@hw.ac.uk}). (16/10/2007)} \and Anke Wiese$^\ast$}

\begin{document}
\maketitle
\label{firstpage}

\begin{abstract}
We present Lie group integrators for nonlinear 
stochastic differential equations with non-commutative vector fields
whose solution evolves on a smooth finite dimensional manifold.
Given a Lie group action that generates transport along the manifold,
we pull back the stochastic flow on the manifold to the Lie group
via the action, and subsequently pull back the flow 
to the corresponding Lie algebra via the exponential map.
We construct an approximation to the stochastic flow in the Lie algebra
via closed operations and then push back to the Lie group and 
then to the manifold, thus ensuring our approximation lies in 
the manifold. We call such schemes stochastic Munthe-Kaas methods
after their deterministic counterparts. We also  
present stochastic Lie group integration schemes based 
on Castell--Gaines methods. These involve using an 
underlying ordinary differential integrator
to approximate the flow generated by a truncated 
stochastic exponential Lie series. They 
become stochastic Lie group integrator schemes if we use
Munthe-Kaas methods as the underlying ordinary differential 
integrator. Further, we show that some Castell--Gaines methods are 
uniformly more accurate than the corresponding 
stochastic Taylor schemes. Lastly we demonstrate
our methods by simulating the dynamics of a free rigid body
such as a satellite and an autonomous underwater vehicle both 
perturbed by two independent multiplicative stochastic noise processes.
\end{abstract}



\pagestyle{myheadings}
\thispagestyle{plain}
\markboth{Malham and Wiese}{Stochastic Lie group integrators}

\section{Introduction}
We are interested in designing Lie group numerical schemes
for the strong approximation 
of nonlinear Stratonovich stochastic differential
equations of the form
\begin{equation}\label{sde}
y_t=y_0+\sum_{i=0}^d\int_0^t V_i(y_\tau,\tau)\,\mathrm{d}W_\tau^i\,.
\end{equation}
Here $W^1,\ldots,W^d$ are $d$ independent scalar Wiener processes
and $W_t^0\equiv t$. We suppose that the solution $y$ 
evolves on a smooth $n$-dimensional submanifold $\mathcal M$ of $\mathbb R^N$ 
with $n\leq N$ and
$V_i\colon\M\times\R_+\rightarrow T\mathcal{M}$, 
$i=0,1,\ldots,d$,
are smooth vector fields which in local coordinates are 
$V_i=\sum_{j=1}^n V_i^j\partial_{y_j}$.
The flow-map $\varphi_t\colon\mathcal M\rightarrow\mathcal M$ 
of the integral equation~\eqref{sde} is defined as the map taking the
initial data $y_0$ to the solution $y_t$ at time~$t$, 
i.e.\ $y_t=\varphi_t\circ y_0$. 

Our goal in this paper is to show how the Lie group integration
methods developed by Munthe-Kaas and co-authors can be extended
to stochastic differential equations on smooth manifolds
(see Crouch and Grossman~\cite{CG} and Munthe-Kaas~\cite{MK}).
Suppose we know that the exact solution of a given system of
stochastic differential equations evolves on a smooth
manifold $\M$ (see Malliavin~\cite{Malliavin} or Emery~\cite{Emery}), 
but we can only find the solution
pathwise numerically. How can we ensure that our 
approximate numerical solution also lies in the manifold? 

Suppose we are given a finite dimensional Lie group $\G$
and Lie group action $\Lambda_{y_0}$ that generates transport across 
the manifold $\M$ from the starting point $y_0\in\M$ 
via elements of $\G$. Then with any given elements $\xi$
in the Lie algebra $\g$ corresponding to the Lie group $\G$,
we can associate the infinitesimal action 
$\lambda_\xi$ using the Lie group action $\Lambda_{y_0}$.
The map $\xi\mapsto\lambda_\xi$ is a Lie algebra 
homomorphism from $\g$ to $\X(\M)$, the Lie algebra
of vector fields over the manifold $\M$. Further 
the Lie subalgebra $\{\lambda_\xi\in\X(\M)\colon\xi\in\g\}$ 
is isomorphic to a finite dimensional Lie algebra with the 
same structure constants (see Olver~\cite{O2}, p.~56).

Conversely, suppose we know that the Lie algebra generated by
the set of governing vector fields $V_i$, $i=0,1,\ldots,d$,
on $\M$ is finite dimensional, call this $\X_F(\M)$. Then
we know there exists a finite dimensional Lie group $\G$
whose Lie algebra $\g$ has the same structure constants 
as $\X_F(\M)$ relative to some basis, and there is a 
Lie group action $\Lambda_{y_0}$ such that 
$V_i=\lambda_{\xi_i}$, $i=0,1,\ldots,d$, for some $\xi_i\in\g$
(see Olver~\cite{O2}, p.~56 or Kunita~\cite{Ku1990}, p.~194).
The choice of group and action is not unique. 

In this paper we assume that there is a finite dimensional
Lie group $\G$ and action $\Lambda_{y_0}$ such that
our set of governing vector fields $V_i$, $i=0,1,\ldots,d$, 
are each infinitesimal Lie group actions generated by
some element in $\g$ via $\Lambda_{y_0}$, 
i.e.\ $V_i=\lambda_{\xi_i}$ for some 
$\xi_i\in\g$, $i=0,1,\ldots,d$. 
They are said to be fundamental vector fields.
This means that we can write down the set of 
governing vector fields $X_{\xi_i}$ for a system 
of stochastic differential equations on the 
Lie group $\G$ that, via the Lie group action $\Lambda_{y_0}$, 
generates the flow governed by the set of 
vector fields $V_i$ on the manifold. 
The vector fields $V_i$ on $\M$ are simply the push forward
of the vector fields $X_{\xi_i}$ on $\G$ via the 
Lie group action $\Lambda_{y_0}$.
Typically the flow on the Lie group also needs 
to be computed numerically. We thus 
want the approximation to remain
in the Lie group so that the Lie group action 
takes us back to the manifold.

To achieve this, we pull back the set of 
governing vector fields 
$X_{\xi_i}$ on $\G$ to the set of governing 
vector fields $v_{\xi_i}$ on $\g$, via the 
exponential map `$\exp$' from $\g$ to $\G$.
Thus the stochastic flow generated on $\g$ 
by the vector fields $v_{\xi_i}$ generates the stochastic flow 
on $\G$ generated by the $X_{\xi_i}$. The 
set of governing vector fields on $\g$ are 
for each $\sigma\in\g$:
\begin{equation}\label{vexplicit}
v_{\xi_i}\circ\sigma\equiv
\sum_{k=0}^\infty \frac{B_k}{k!}\,(\mathrm{ad}_\sigma)^k\circ\xi_i\,,
\end{equation}
where $B_k$ is the $k$th Bernoulli number and 
the adjoint operator $\mathrm{ad}_\sigma$ 
is a closed operator on $\g$, in fact 
$\mathrm{ad}_\sigma\circ\zeta=[\sigma,\zeta]$, 
the Lie bracket on $\g$. Now the essential point is
that $\xi_i\in\g$ and so the 
series on the right or any truncation of it is
closed in $\g$. Hence if we construct an approximation
to our stochastic differential equation on $\g$ using
the vector fields $v_{\xi_i}$ or an approximation of them
achieved by truncating the series representation, then
that approximation must reside in the Lie algebra $\g$.
We can then push the approximation
in the Lie algebra forward onto the Lie group and then onto the
manifold. Provided we compute the exponential map 
and action appropriately, our approximate solution
lies in the manifold (to within machine accuracy).
In summary, for a given $\xi\in\g$ and any $y_0\in\M$ we have
the following commutative diagram:
\begin{equation*}
\begin{CD} 
\g @ >\exp_\ast>>  \X(\G)@ >(\Lambda_{y_0})_\ast>> \X(\M) \\ 
   @ A v_\xi AA   @ AA X_\xi A @ AA \lambda_\xi A  \\ 
\g @ >\exp>>       \G    @ >\Lambda_{y_0}>>  \M      
\end{CD} 
\end{equation*}

We have implicitly separated the governing set of vector fields 
$V_i$, $i=0,1,\ldots,d$, from the driving path process 
$w\equiv(W^1,\ldots,W^d)$. Together they generate
the unique solution process $y\in\M$ 
to the stochastic differential equation~\eqref{sde}.
When there is only one driving Wiener process ($d=1$) 
the It\^o map $w\mapsto y$ is
continuous in the topology of uniform convergence.
When there are two or more driving processes ($d\geq2$)
the Universal Limit Theorem tells us that
the It\^o map $w\mapsto y$ is continuous in the 
$p$-variation topology, in particular for 
$2\leq p<3$ (see Lyons~\cite{L}, Lyons and Qian~\cite{LQ} and 
Malliavin~\cite{Malliavin}). A Wiener
path with $d\geq2$ has finite $p$-variation for $p>2$.
This means that from a pathwise perspective, 
approximations to $y$ constructed using 
successively refined approximations to $w$ 
are only guaranteed to converge to the correct
solution $y$, if we include information about
the L\'evy chordal areas of the driving path
process. Note however that the $L^2$-norm of
the $2$-variation of a Wiener process is finite. 
In the Lie group integration procedure 
prescribed above we must solve
a stochastic differential system on the 
Lie algebra $\g$ defined by the set of 
governing vector fields $v_{\xi_i}$ 
and the driving path process $w\equiv(W^1,\ldots,W^d)$.
In light of the Universal Limit Theorem and 
with stepsize adaptivity in mind in future
(see Gaines and Lyons~\cite{GL}), 
we for instance use in our examples order~$1$ stochastic 
numerical methods---that include
the L\'evy chordal area---to solve for the flow
on the Lie algebra $\g$.

We have thus explained the idea behind Munthe-Kaas methods
and how they can be generalized to the stochastic setting. 
The first half of this paper formalizes this procedure.

In the second half of this paper, we 
consider autonomous vector fields and 
construct stochastic Lie group integration schemes 
using Castell--Gaines methods. This approach
proceeds as follows. We truncate the 
stochastic exponential Lie series expansion
corresponding to the flow $\varphi_t$ of 
the solution process $y$ to the 
stochastic differential equation~\eqref{sde}.
We then approximate the driving path process 
$w\equiv(W^1,\ldots,W^d)$ by replacing it by
a suitable nearby piecewise smooth path
in the appropriate variation topology.
An approximation to
the solution $y_t$ requires the exponentiation
of the approximate truncated exponential Lie series.
This can be achieved by solving the system of 
ordinary differential equations driven by the
vector field that is the  
approximate truncated exponential Lie series.
If we use ordinary Munthe-Kaas methods 
as the underlying ordinary differential integrator
the Castell--Gaines method becomes a stochastic
Lie group integrator. 

Further, based on the Castell--Gaines approach we then present 
\emph{uniformly accurate exponential Lie series integrators} 
that are globally more accurate
than their stochastic Taylor counterpart schemes
(these are investigated in detail in 
Lord, Malham and Wiese~\cite{LMW} 
for linear stochastic differential equations). They
require the assumption that a sufficiently accurate
underlying ordinary differential integrator is used;
that integrator could for example be an ordinary
Lie group Munthe-Kaas method.
In the case of two driving Wiener processes we derive  
the order~$1/2$, and in 
the case of one driving Wiener process the order~$1$
uniformly accurate exponential Lie series integrators.
As a consequence we confirm the asymptotic efficiency properties 
for both these schemes
proved by Castell and Gaines~\cite{CG} (see Newton~\cite{N}
for more details on the concept of asymptotic efficiency). 
We also present in the case of one driving Wiener process 
a new order~$3/2$ uniformly accurate exponential Lie series 
integrator (also see Lord, Malham and Wiese~\cite{LMW}).

We present two physical applications that demonstrate
the advantage of using stochastic Munthe-Kaas methods.
First we consider a free rigid body which for example
could model the dynamics of a satellite. We suppose 
that it is perturbed by two independent multiplicative 
stochastic noise processes. The governing vector fields
are non-commutative and the corresponding exact stochastic
flow evolves on the unit sphere. 
We show that the stochastic Munthe-Kaas method,
with an order~$1$ stochastic Taylor integrator used to 
progress along the corresponding 
Lie algebra, preserves the approximate solution in the 
unit sphere manifold to within machine error. However
when an order~$1$ stochastic Taylor integrator is used
directly, the solution leaves the unit sphere. 
The contrast between these two methods is more emphatically 
demonstrated in our second application. Here we consider
an autonomous underwater vehicle that is also
perturbed by two independent multiplicative 
stochastic noise processes. The exact stochastic flow
evolves on the manifold which is the dual of the 
Euclidean Lie algebra $\mathfrak{se}(3)$; 
two independent Casimirs are
conserved by the exact flow. Again the stochastic
Munthe-Kass method preserves the Casimirs to within
machine error. However the order~$1$ stochastic Taylor 
integrator is not only unstable for
large stepsizes, but the approximation drifts off the manifold
and makes a dramatic excursion off to infinity in the
embedding space $\R^6$.

Preserving the approximate flow on the manifold of the
exact dynamics may be a required property 
for physical or financial systems driven by smooth
or rough paths---for general references see 
Iserles, Munthe-Kaas, N\o rsett and Zanna~\cite{IMNZ},
Hairer, Lubich and Wanner~\cite{HLW}, Elworthy~\cite{E},
Lyons and Qian~\cite{LQ}
and Milstein and Tretyakov~\cite{MT}. Stochastic Lie group 
integrators in the form of Magnus integrators for linear
stochastic differential equations were investigated  
by Burrage and Burrage~\cite{BB}. They were also
used in the guise of M\"obius schemes (see Schiff and Shnider~\cite{SS}) 
to solve stochastic Riccati equations by 
Lord, Malham and Wiese~\cite{LMW} where they outperformed
direct stochastic Taylor methods.
Further applications where they might be applied include: 
backward stochastic Riccati equations arising in
optimal stochastic linear-quadratic control 
(Kohlmann and Tang~\cite{KT}); jump diffusion 
processes on matrix Lie groups for Bayesian inference 
(Srivastava, Miller and Grenander~\cite{SMG}); 
fractional Brownian motions on Lie groups (Baudoin and Coutin~\cite{BC}) and
stochastic dynamics triggered by DNA damage (Chickarmane,
Ray, Sauro and Nadim~\cite{CRSN}).

Our paper is outlined as follows. In Section~2 we 
present the basic geometric setup, \emph{sans} stochasticity.
In particular we present a generalized right translation vector field 
on a Lie group that forms the basis of our subsequent
transformation from the Lie group to the manifold.
Using a Lie group action, this vector field pushes forward 
to an infinitesimal Lie group action vector field 
that generates a flow on the smooth manifold.
In Section~3 we specialize to the case of a matrix Lie group
and using the exponential map, derive the pullback
of the generalized right translation vector field 
on the Lie group to the corresponding vector field
on the Lie algebra.
To help give some context to our overall scheme, we 
provide in Section~4 illustrative examples
of manifolds and natural choices for associated 
Lie groups and actions that generate flows on those
manifolds. Then in Section~5 we show how a flow 
on a smooth manifold corresponding to a 
stochastic differential equation can be generated 
by a stochastic flow on a Lie algebra 
via a Lie algebra action.
We explicitly present stochastic Munthe-Kaas
Lie group integration methods in Section~6. 
We start the second half of our paper by reviewing the exponential
Lie series for stochastic differential equations in Section~7. We 
show in Section~8 how to construct geometric stochastic Castell--Gaines 
numerical methods. In particular we also present 
uniformly accurate exponential Lie series numerical schemes 
that not only can be used as geometric stochastic integrators,
but also are always more accurate than stochastic Taylor 
numerical schemes of the corresponding order.
In Section~9 we present our concrete numerical examples.
Finally in Section~10 we conclude and present some 
further future applications and directions.

\section{Lie group actions}\label{Liegroupactions}
Suppose $\M$ is a smooth finite $n$-dimensional 
submanifold of $\R^N$ with $n\leq N$. We use $\X(\M)$ to denote 
the Lie algebra of vector fields on the manifold $\M$,
equipped with the Lie--Jacobi bracket
$[U,V]\equiv U\cdot\nabla V-V\cdot\nabla U$, for
all $U,V\in\X(\M)$.
Let $\G$ denote a finite dimensional Lie group.%\smallskip

\begin{definition}[Lie group action]
A left Lie group \emph{action} of a Lie group $\mathcal G$ on a manifold $\mathcal M$
is a smooth map 
$\Lambda\colon \mathcal G\times\mathcal M\rightarrow\mathcal M$ satisfying
for all $y\in\mathcal M$ and $R,S\in\mathcal G$: 
(1) $\Lambda(\I,y)=y$;
(2) $\Lambda(R,\Lambda(S,y))=\Lambda(RS,y)$.
We denote $\Lambda_y\circ S\equiv\Lambda(S,y)$.
\end{definition}%\smallskip

Hereafter we suppose $y_0\in\M$ is fixed and 
focus on the action map $\Lambda_{y_0}\colon\G\ra\M$.
We assume that the Lie group action $\Lambda$ 
is \emph{transitive}, i.e.\ transport across the manifold from
any point $y_0\in\M$ to any other point $y\in\M$ can always
be achieved via a group element $S\in\G$ with
$y=\Lambda_{y_0}\circ S$ (Marsden and Ratiu~\cite{MR}, p.~310).

We define the Lie algebra $\g$ associated with the Lie group $\G$
to be the \emph{vector space of all right invariant vector fields on} $\G$.
By standard construction this is isomorphic to the tangent space to
$\G$ at the identity $\I\equiv\I_{\G}$ (see Olver~\cite{O2}, p.~48 
or Marsden and Ratiu~\cite{MR}, p.~269).%\smallskip

\begin{definition}[Generalized right translation vector field]
Suppose we are given a smooth map $\xi\colon\M\ra\g$.
With each such map $\xi$ we associate
a vector field $X_\xi\colon\G\rightarrow\X(\G)$ defined 
as follows
\begin{equation*}
X_\xi\circ S\equiv\left.\partial_\tau
\exp\bigl(\tau\,\xi(\Lambda_{y_0}\circ S)\bigr)\,S\right|_{\tau=0}
\end{equation*}
for $S\in\G$, where `$\exp$' is the usual local diffeomorphism 
$\exp\colon\g\rightarrow\G$ from a neighbourhood 
of the zero element $\oo\in\g$ to
a neighbourhood of $\I\in\G$.
\end{definition}
%\smallskip

\begin{definition}[Infinitesimal Lie group action]
We associate with each vector field $X_\xi\colon\G\ra\X(\G)$
a vector field $\lambda_\xi\colon\M\ra\X(\M)$
as the push forward of $X_\xi$ from $\G$ to $\M$
by $\Lambda_{y_0}$, i.e.
$\lambda_\xi\equiv\bigl(\Lambda_{y_0}\bigr)_\ast X_\xi$,
so that if $S\in\G$ and $y=\Lambda_{y_0}\circ S\in\M$, then
\begin{equation*}
\lambda_\xi\circ y\equiv\left.\partial_\tau 
\Lambda_{y_0}\circ\gamma(\tau)\right|_{\tau=0}\,,
\end{equation*}
where $\gamma(t)\in\G$, $\gamma(0)=S$ and 
$\partial_\tau\gamma(\tau)=X_\xi\circ\gamma(\tau)$
(the flow generated on $\G$ by the vector field $X_\xi$ starting
at $S\in\G$). Naturally, as a vector field $\lambda_\xi$ is linear, 
and also
\begin{equation*}
\lambda_\xi\circ y\equiv \mathcal L_{X_\xi}\circ\Lambda_{y_0}\circ S\,,
\end{equation*}
the Lie derivative of $\Lambda_{y_0}$ along $X_\xi$ at $S\in\G$.
\end{definition}%\smallskip

\emph{Remarks.}
\begin{remunerate}
\item The map $\Lambda(S)\colon\M\ra\M$ defined
by $y\mapsto\Lambda(S)\circ y\equiv\Lambda_y\circ S$ represents 
a flow on $\M$. Hence if $y=\Lambda(S)\circ y_0$, the push forward
of $\lambda_\xi$ by $\Lambda(S)$ is given by  
$\bigl(\Lambda(S)\bigr)_\ast\lambda_\xi\equiv\lambda_{\Ad_S\xi}$
(Marsden and Ratiu~\cite{MR}, p.~317).
\item We define the \emph{isotropy subgroup} at $y_0\in\M$
by $\G_{y_0}\equiv\{S\in\G\colon\Lambda_{y_0}\circ S=y_0\}$;
it is a closed subgroup of $\G$ (see Helgason~\cite{Helgason}, p.~121
or Warner~\cite{W}, p.~123). We define the \emph{global isotropy subgroup}
by $\G_{\M}\equiv\cap_{y_0\in\M}\G_{y_0}\equiv\{S\in\G\colon\Lambda_{y_0}\circ
S=y_0,~\forall y_0\in\M\}$; it is a normal subgroup of $\G$
(see Olver~\cite{O2}, p.~38).
\item A Lie group action is said to be is \emph{effective/faithful} 
if the map $S\mapsto\Lambda(S)$ from $\G$ to $\text{Diff}(\M)$,
the group of diffeomorphisms on $\M$, is one-to-one. This is equivalent
to the condition that different group elements have different
actions, i.e.\ $\G_{\M}\equiv\{\I_\G\}$. A Lie group action
is said to be \emph{free} if $\G_{y_0}=\{\I_\G\}$ for all $y_0\in\M$,
i.e.\ $\Lambda_{y_0}$ is a diffeomorphism from $\G$ to $\M$.
For more details see Marsden and Ratiu~\cite{MR}, p.~310 and
Olver~\cite{O2}, p.~38.
\item The map $\gamma\colon\G/\G_{y_0}\ra\M$ defined by $\gamma\colon
S\cdot\G_{y_0}\mapsto\Lambda_{y_0}\circ S$ is a diffeomorphism,
i.e.\ $\M\cong\G/\G_{y_0}$ for any $y_0\in\M$
(a manifold $\M$ with a Lie group action $\Lambda\colon\G\times\M\ra\M$
defined over it is thus diffeomorphic to a \emph{homogeneous manifold};
see Warner~\cite{W}, p.~123 or Olver~\cite{O2}, p.~40).
Further, the induced action of $\G/\G_{\M}$ on $\M$ is effective.
Hence if $\Lambda$ is not an effective action of $\G$, 
we can replace it (without loss of generality) 
by the induced action of $\G/\G_{\M}$ (see Olver~\cite{O2}, p.~38).
\item Our definition for the 
generalized right translation vector field $X_\xi$
on $\G$ is motivated by the standard right translation vector field used to
identify $\g$, the vector space of right invariant vector fields on $\G$, 
with $T_{\I}\G$, the tangent space to $\G$ at the identity. When 
$\xi\in\g$ is constant, $X_\xi\in\X(\G)$ is right invariant and 
a Lie bracket on $T_\I\G$ can be defined via right extension by
the corresponding Lie--Jacobi bracket for the vector fields $X_\xi$
on $\X(\G)$. Unless $\xi\in\g$ is constant, $X_\xi$ is not in 
general right invariant. For further details see Varadarajan~\cite{V}, 
Olver~\cite{O2}, or Marsden and Ratiu~\cite{MR}.
\item The infinitesimal generator map $\xi\mapsto\lambda_\xi$ 
from $\g$ to $\X(\M)$ is a Lie algebra homomorphism. 
If we identify $\g$ as the vector space of left invariant
vector fields on $\G$ this map becomes an anti-homomorphism.
The Lie--Jacobi bracket as defined above gives the right (rather than left)
Lie algebra stucture over the group of diffeomorphisms on $\M$. 
If in addition we take the Lie--Jacobi bracket to be minus that 
defined above---associated with the left Lie algebra structure---then
the infinitesimal generator map becomes a homomorphism again. See for
example Marsden and Ratiu~\cite{MR}, p.~324 or Munthe-Kaas~\cite{MK}.
\item The image of $\g$ under the infinitesimal generator map
$\xi\mapsto\lambda_\xi$ forms a finite dimensional Lie algebra
of vector fields on $\M$ which is isomorphic to the Lie algebra
of the effectively acting quotient group $\G/\G_{\M}$ 
(see Olver~\cite{O2}, p.~56). 
Thus the tangent space to $\M$ at any point is $\g$ and
$\M$ inherents a connection from $\G/\G_\M$.
Connections are necessary to define
martingales on manifolds, but not for
defining semimartingales (our focus here);
see Malliavin~\cite{Malliavin} and Emery~\cite{Emery}.
\item A comprehensive study of the systematic construction
of symmetry Lie groups from given vector fields can be found in Olver~\cite{O2}. 
\item We assumed above that the 
vector fields $X_\xi$ and $\lambda_\xi$ are autonomous.
However all results in this and subsequent sections up
to Section~\ref{els} can be straightforwardly extended
to non-autonomous vector fields generated by 
$\xi\colon\M\times\R\ra\g$ with 
$(y,t)\mapsto\xi(y,t)$ for all $y\in\M$ and $t\in\R$.
\item For full generality we want to suspend reference to embedding
spaces as far as possible. However in subsequent sections to be concise
we will more explicitly reclaim this context.
\end{remunerate}

\section{Pull back to the Lie algebra}
For ease of presentation, we will assume in this 
section that $\G$ is a matrix Lie group.
Recall that the exponential map 
$\exp\colon\g\rightarrow\G$ is a local diffeomorphism
from a neighbourhood of $\oo\in\g$ to a neighbourhood of 
$\I\in\G$.
Let $v_\xi\colon\g\ra\g$ 
be the pull back of the vector field 
$X_{\xi}\colon\G\ra\X(\G)$
from $\G$ to $\g$ via the exponential mapping 
$\exp\colon\g\ra\G$, i.e.\ 
$v_\xi\circ\sigma\equiv\exp^\ast X_{\xi}\circ\sigma$.
If $\sigma\in\g$ then  
\begin{equation}
v_\xi\circ\sigma=\mathrm{dexp}_\sigma^{-1}\circ
\xi\bigl(\Lambda_{y_0}\circ\exp\sigma\bigr)\,.
\label{pullbacktog}
\end{equation}
Here $\mathrm{dexp}^{-1}_\sigma\colon\g\ra\g$ is the inverse
of the right-trivialized tangent map of the exponential
$\mathrm{dexp}_\sigma\colon\g\ra\g$ defined as follows.
If $\beta(\tau)$ is a curve in $\g$ such that $\beta(0)=\sigma$
and $\beta'(0)=\eta\in\g$ then 
$\mathrm{dexp}\colon\g\times\g\ra\g$ 
is the local smooth map (Varadarajan~\cite{V}, p.~108)
\begin{align*}
\mathrm{dexp}_\sigma\circ\eta
&\equiv\left.\partial_\tau\exp\beta(\tau)\right|_{\tau=0}\,\exp(-\sigma)\\
&=\left(\frac{\exp(\mathrm{ad}_\sigma)-\I}{\mathrm{ad}_\sigma}\right)\circ\eta\,.
\end{align*}
Note that as a tangent map $\mathrm{dexp}_\sigma\colon\g\ra\g$ is linear.
The inverse operator $\mathrm{dexp}^{-1}_\sigma$ is 
the operator series~\eqref{vexplicit} 
generated by considering the reciprocal of $\mathrm{dexp}_\sigma$.

To show that~\eqref{pullbacktog} is true, if 
$\exp\colon\g\ra\G$ with $\sigma\mapsto S=\exp\sigma$, and 
$\beta(\tau)\in\g$ with $\beta(0)=\sigma$ and 
$\partial_\tau\beta(\tau)=v_\xi\circ\beta(\tau)$, then:
\begin{align*} 
\exp_\ast v_{\xi}\circ S
=&\;\left.\partial_\tau\exp\beta(\tau)\right|_{\tau=0}\\
=&\;\bigl(\mathrm{dexp}_\sigma \circ v_\xi\circ\sigma\bigr)\,\exp(\sigma)\\
\equiv&\;X_\xi\circ S\,.
\end{align*}
Since `$\exp$' is a diffeomorphism in a neighbourhood
of $\oo\in\g$, this push forward calculation 
establishes the pull back~\eqref{pullbacktog} for all 
$\sigma\in\g$ in that neighbourhood.

\section{Illustrative examples}
Suppose the vector field $V\colon\M\times\R\ra\X(\M)$
generates a flow solution $y_t\in\M$ starting from $y_0\in\M$. 
Then assume there exists a: 
\begin{enumerate}
\item Lie group $\mathcal G$ with corresponding Lie algebra $\g$;
\item Lie group action $\Lambda_{y_0}\colon\G\ra\M$ 
for which a starting point $y_0\in\M$ is fixed;
\item Vector field $\lambda_{\xi}\colon\M\times\R\ra\X(\M)$ such that:
$V\equiv\lambda_{\xi}$, i.e.\ $V$ is a fundamental vector field 
corresponding to the action $\Lambda_{y_0}$. 
\end{enumerate}

Let us suppose $\G$ is a matrix Lie group (or can be embedded into
a matrix Lie group, for example the Euclidean group $SE(3)$ is 
naturally embedded into the special linear group $SL(4;\R)$).
We have for all $S\in\G$ and $t\in\R$,
\begin{equation}\label{GVF}
X_\xi(S,t)\equiv\xi\bigl(\Lambda_{y_0}(S),t\bigr)\,S\,.
\end{equation}
If $V=\lambda_\xi$ for some $\xi\colon\M\ra\g$, some Lie group $\mathcal G$ 
and corresponding action $\Lambda_{y_0}$, then the flow generated by
$X_\xi$ on $\G$ drives the flow generated by $V$ on $\M$.
In each of the examples below, given the manifold $\M$, we 
present a natural Lie group and action associated
with the manifold structure, and identify vector fields 
which generate flows on the manifold via the Lie group.

\subsubsection*{Stiefel manifold $\mathbb V_{n,k}$} Suppose 
$\M=\mathbb V_{n,k}\equiv\{y\in\R^{n\times k}\colon y^{\text{T}}y=I\}$. 
Take $\G=SO(n)$, the special orthogonal group,
and $\Lambda_{y_0}(S)\equiv Sy_0$, the action of left multiplication.
The corresponding Lie algebra $\g=\mathfrak{so}(n)$. 
Then by direct calculation $\lambda_\xi(y)=\xi(y,t)\,y$.
Hence if the given vector field $V(y,t)=\xi(y,t)\,y$,
then the push forward of the flow generated by 
$X_\xi(S,t)$ on $\G$ in~\eqref{GVF} is the 
flow generated by $V$ on $\M$.
Note that the unit sphere $\mathbb S^2\cong\mathbb V_{3,1}$, i.e.\  
$\mathbb S^2$ is just a particular Stiefel manifold. In Section~9
as an application, we consider rigid body dynamics 
evolving on $\mathbb S^2$.

\subsubsection*{Isospectral manifold $\mathbb S_n$} Suppose 
$\M=\mathbb S_n=\{y\in\R^{n\times n}\colon y^{\text{T}}=y\}$,
the set of $n\times n$ real symmetric matrices.
Take $\G=O(n)$, the orthogonal group and 
$\Lambda_{y_0}(S)\equiv Sy_0S^{\text{T}}$, 
which is an isospectral action (Munthe-Kaas~\cite{MK}).
The corresponding Lie algebra is $\g=\mathfrak{so}(n)$.
Again, by direct calculation
$\lambda_\xi(y)=\xi(y,t)\,y-y\,\xi(y,t)$.
Hence if the given vector field 
$V(y,t)=\xi(y,t)\,y-y\,\xi(y,t)$,
then the push forward of the flow generated by 
$X_\xi(S,t)$ on $\G$ in~\eqref{GVF} is the 
flow generated by $V$ on $\M$.

\subsubsection*{Dual of the Euclidean algebra $\mathfrak{se}(3)^\ast$}
Suppose $\M=\mathfrak{se}(3)^\ast\cong\R^3$, 
the dual of the Euclidean algebra $\mathfrak{se}(3)$ 
of the Euclidean group 
$SE(3)=\bigl\{(s,\rho)\in SE(3)\colon s\in SO(3),~\rho\in\R^3\bigr\}$.
Take $\G=SE(3)$ so $\g=\mathfrak{se}(3)$ 
and $\Lambda\equiv\mathrm{Ad}^\ast\colon\G\times\g^\ast\ra\g^\ast$,
the coadjoint action of $\G$ on $\g^\ast$.
Then by direct calculation 
$\lambda_\xi(y)=-\mathrm{ad}_{\xi}^\ast(y)$.
Since $\lambda_\xi(y)$ in linear in $\xi$ and
$-\lambda_\xi(y)\equiv\lambda_{-\xi}(y)$,
it follows that if $V(y)=\mathrm{ad}_{\xi}^\ast(y)$,  
then the push forward of the flow generated by 
$X_{-\xi}(S,t)=-\xi\bigl(\Lambda_{y_0}(S),t\bigr)\,S$ on $\G$ is the 
flow generated by $V$ on $\M$. For more details see
Section~9 where we investigate the dynamics of an
autonomous underwater vehicle 
evolving on $\mathfrak{se}(3)^\ast$.

%Note that coadjoint orbits of a Lie group are 
%symplectic manifolds (Marsden and Ratiu~\cite{MR}, p.~445).

\subsubsection*{Grassmannian manifold $\mathrm{Gr}(k,n)$}
The Grassmannian manifold $\mathcal M=\text{Gr}(k,n)$ is
the space of $k$-dimensional subspaces of $\mathbb R^n$.
Take $\mathcal G=\text{GL}(n)$, the general linear matrix group,
where if $S\in\text{GL}(n)$, we identify
\begin{equation*}
S=\begin{pmatrix}\alpha & \beta\\\gamma &\delta \end{pmatrix}\,,
\end{equation*}
where the block matrices $\alpha$, $ \beta$, $\gamma$ 
and $\delta$ are sizes $k\times k$, $k\times(n-k)$, 
$(n-k)\times k$ and $(n-k)\times(n-k)$, respectively
(see Schiff and Shnider~\cite{SS}; Munthe-Kaas~\cite{MK}).
We choose the action of $\text{GL}(n)$ on $\text{Gr}(k,n)$ 
to be the generalized M\"obius transformation 
$\Lambda_{y_0}(S)=(\alpha y_0+\beta)(\gamma y_0+\delta)^{-1}$.
Hence if
\begin{equation*}
\xi(t)=\begin{pmatrix}a(t)&b(t)\\c(t)&d(t)\end{pmatrix}\,,
\end{equation*}
then direct calculation reveals that
$\lambda_\xi(y)=a(t)y+b(t)-yc(t)y-yd(t)$.
Hence if the given vector field $V(y)=a(t)y+b(t)-yc(t)y-yd(t)$,
then the push forward of the flow generated by 
$X_\xi(S,t)=\xi(t)\,S$ on $\G$ is the 
flow generated by $V$ on $\text{Gr}(k,n)$.

\section{Stochastic Lie group integration}\label{cgi}
We show that if a Lie group action $\Lambda\colon\G\times\M\ra\M$
exists, then for $y_0\in\M$ fixed, the Lie algebra action
$\Lambda_{y_0}\circ\exp\colon\g\ra\M$ carries a flow on $\g$ to 
a flow on $\M$.

\begin{theorem}\label{pullbacklemma}
Suppose there exists a Lie group action
$\Lambda\colon\G\times\M\ra\M$. Then if 
there exists a process 
$\sigma\in\g$ and a stopping time $T_*$ 
such that on $[0,T_*)$,
$\sigma$ satisfies the Stratonovich
stochastic differential equation
\begin{equation}\label{sigmaeqn}
\sigma_t=\sum_{i=0}^d\int_0^t
v_{\xi_i}\circ\sigma_\tau\,\mathrm{d}W_\tau^i\,,
\end{equation}
then the process $y=\Lambda_{y_0}\circ\exp\sigma\in\M$ 
satisfies the Stratonovich 
stochastic differential equation on $[0,T_*)$:
\begin{equation}\label{Msde}
y_t=y_0+\sum_{i=0}^d\int_0^t\lambda_{\xi_i}\circ y_\tau\,\mathrm{d}W_\tau^i\,.
\end{equation}
\end{theorem}

\begin{proof}
Using It\^o's lemma, if $\sigma_t\in\g$ satisfies~\eqref{sigmaeqn}
then $\Lambda_{y_0}\circ\exp\sigma_t$ satisfies
\begin{equation*}
\Lambda_{y_0}\circ\exp\sigma_t=\Lambda_{y_0}\circ\exp\oo
+\sum_{i=0}^d\int_0^t\mathcal L_{v_{\xi_i}}\circ
\Lambda_{y_0}\circ\exp\sigma_\tau
\,\mathrm{d}W^i_\tau\,.
\end{equation*}
Now recall that for each $i=0,1,\ldots,d$, 
$X_{\xi_i}$ is the push forward of
$v_{\xi_i}$ from $\g$ to $\G$ via the exponential map,
and that $\lambda_{\xi_i}$ is the push forward of 
$X_{\xi_i}$ from $\G$ to $\M$ via $\Lambda_{y_0}$ and
so the Lie derivative
\begin{equation*}
\mathcal L_{v_{\xi_i}}\circ\Lambda_{y_0}\circ\exp\sigma_t
\equiv\lambda_{\xi_i}\circ y_t\,.
\end{equation*}
Then since $y_t=\Lambda_{y_0}\circ\exp\sigma_t$, we conclude 
that $y\in\M$ is a process satisfying the 
stochastic differential equation~\eqref{Msde}.\qquad\end{proof}
%\smallskip

\begin{corollary}\label{cor}
Suppose that for each $i=0,1,\ldots,d$  
there exists $\xi_i\colon\M\ra\g$
such that the vector field $V_i\colon\M\ra\X(\M)$ and 
$\lambda_{\xi_i}\colon\M\ra\X(\M)$ can be identified, 
i.e.\
\begin{equation}
V_i\equiv\lambda_{\xi_i}\,.
\label{suitableform}
\end{equation}
Then the push forward by `$\Lambda_{y_0}\circ\exp$' of the flow on  
the Lie algebra manifold $\g$ generated by the  
stochastic differential equation~\eqref{sigmaeqn} is
the flow on the smooth manifold $\M$ generated by
the stochastic differential equation~\eqref{Msde},
whose solution can be expressed in the form 
$y_t=\Lambda_{y_0}\circ\exp\sigma_t$.
\end{corollary}

\emph{Remark.} If the action is free then `$\Lambda_{y_0}\circ\exp$'
is a diffeomorphism from a neighbourhood of
$\oo\in\g$ to a neighbourhood of $y_0\in\M$.

\section{Stochastic Munthe-Kaas methods}
Assuming that the vector fields in our original stochastic
differential equation~\eqref{sde} are fundamental and 
satisfy \eqref{suitableform}, then 
stochastic Munthe-Kaas methods are constructed as follows:
\begin{enumerate}
\item Subdivide the global interval of integration
$[0,T]$ into subintervals $[t_n,t_{n+1}]$.
\item Starting with $t_0=0$, repeat the next two steps 
over successive intervals $[t_n,t_{n+1}]$ until $t_{n+1}=T$.
\item Compute an approximate solution $\hat\sigma_{t_n,t_{n+1}}$ 
to~\eqref{sigmaeqn} across $[t_n,t_{n+1}]$ 
using a stochastic Taylor, stochastic Runge--Kutta or Castell--Gaines method.
\item Compute the approximate solution 
$y_{t_{n+1}}\approx\Lambda_{y_{t_n}}\circ\exp\hat\sigma_{t_n,t_{n+1}}$.
\end{enumerate}

Note that by construction $\hat\sigma_{t_n,t_{n+1}}\in\g$ 
because the stochastic differential equation~\eqref{sigmaeqn}
(or any stochastic Taylor or other sensible approximation)
evolves the solution locally on the Lie algebra~$\mathfrak g$
via the vector fields $v_{\xi_i}\colon\g\ra\g$. 
Suitable methods for approximating the exponential map 
to ensure it maps $\g$ to $\G$ appropriately can be found in 
Iserles and Zanna~\cite{IZ}. Then by 
construction $y_{t_{n+1}}\in\M$.

For example, with two Wiener processes and 
autonomous vector fields $v_{\xi_i}\circ\sigma$,
an order~$1$ stochastic Taylor Munthe-Kaas method is based on 
\begin{equation}\label{STMK}
\hat\sigma_{t_n,t_{n+1}}=\bigl(J_0v_{\xi_0}+J_1v_{\xi_1}+J_2v_{\xi_2}
+\tfrac12J_1^2v_{\xi_1}^2+J_{12}v_{\xi_1}v_{\xi_2}
+J_{21}v_{\xi_2}v_{\xi_1}+\tfrac12J_2^2v_{\xi_2}^2\bigr)\circ\oo\,,
\end{equation} 
evaluated at the zero element $\oo\in\g$.
Typically `$\mathrm{dexp}_\sigma^{-1}$' is truncated 
to only include the necessary low order terms to
maintain the order of the numerical scheme.
%\smallskip

\emph{Remark.} 
It is natural to invoke Ado's Theorem 
(see for example Olver~\cite{O2} p.~54): any finite dimensional
Lie algebra is isomorphic to a Lie subalgebra of $\mathfrak{gl}(n)$
(the general linear algebra) for some $n\in\mathbb N$. However as
Munthe-Kaas~\cite{MK} points out, directly using a matrix representation
for the given Lie group might not lead to the optimal computational
implementation (other data structures might do so).

\section{Exponential Lie series}\label{els}
The stochastic Taylor series is known in different contexts 
as the \emph{Neumann series}, \emph{Peano--Baker series} or
\emph{Feynman--Dyson path ordered exponential}. 
If the vector fields in the stochastic differential equation~\eqref{sde}
are autonomous (which we assume henceforth), i.e.\  for all $i=0,1,\ldots,d$, 
$V_i=V_i(y)$ only, then the stochastic Taylor series for the flow is
\begin{equation*}
\varphi_t=
\sum_{m=0}^\infty\,\sum_{\alpha\in\mathbb P_m}
J_{\alpha_1\cdots \alpha_m}(t)\,V_{\alpha_1}\cdots V_{\alpha_m}\,.
\end{equation*}
Here $\mathbb P_m$ is the set of all combinations
of multi-indices $\alpha=(\alpha_1,\ldots,\alpha_m)$ of length $m$
with $\alpha_i\in\{0,1,\ldots,d\}$ and 
\begin{equation*}
J_{\alpha_1\cdots \alpha_m}(t)\equiv\int_0^t\cdots\int_0^{\tau_{m-1}}
\mathrm{d}W^{\alpha_1}_{\tau_m}\,\cdots\,\mathrm{d}W^{\alpha_m}_{\tau_1}
\end{equation*}
are multiple Stratonovich integrals. 

The logarithm of $\varphi_t$ is the 
\emph{exponential Lie series},
\emph{Magnus expansion} (Magnus~\cite{Magnus}) 
or \emph{Chen--Strichartz formula} (Chen~\cite{Ch}, 
Strichartz~\cite{S}).
In other words we can express the flow map in the form 
$\varphi_t=\exp\psi_t$, where 
\begin{equation*}
\psi_t=\sum_{i=0}^d J_i(t)V_i+\sum_{j>i=0}^d
\tfrac12(J_{ij}-J_{ji})(t)[V_i,V_j]+\cdots
\end{equation*}
is the exponential Lie series for our system,
and $[\cdot\,,\cdot]$ is the Lie--Jacobi bracket on $\X(\mathcal M)$.
See Yamato~\cite{Y}, Kunita~\cite{Ku1980}, 
Ben Arous~\cite{BA} and Castell~\cite{C}
for the derivation and convergence of 
the exponential Lie series expansion in the stochastic context; 
Strichartz~\cite{S} for the full explicit expansion;
Sussmann~\cite{Su} for a related product expansion and 
Lyons~\cite{L} for extensions to rough paths. 

Let us denote the truncated exponential Lie series by
\begin{equation}\label{eq:orderm}
\hat\psi_t=\sum_{\alpha\in\mathbb Q_m}J_\alpha\, c_\alpha\,,
\end{equation}
where $\mathbb Q_m$ denotes the finite set of multi-indices $\alpha$
for which $\|J_\alpha\|_{L^2}$ is of order up to and including $t^m$, 
where $m=1/2,1,3/2,\ldots$. 
The terms $c_\alpha$ are linear combinations 
of finitely many (length $\alpha$)
products of the smooth vector fields $V_i$, $i=0,1,\ldots,d$. The following
asymptotic convergence result can be established along the lines
of the proof for linear stochastic differential equations 
in Lord, Malham and Wiese~\cite{LMW};
we provide a proof in Appendix~\ref{firstbigproof}.%\smallskip

\begin{theorem}\label{th:conv}
Assume the vector fields $V_i$ have 
$2m+1$ uniformly bounded derivatives,
for all $i=0,1,\ldots,d$.
Then for $t\leq1$, the flow $\exp\hat\psi_t\circ y_0$ 
is square-integrable, where 
$\hat\psi_t$ is the truncated Lie series~\eqref{eq:orderm}. 
Further, if $y$ is the solution 
of the stochastic differential equation~\eqref{sde},
there exists a constant $C\bigl(m,\|y_0\|_2\bigr)$ such that
\begin{equation}\label{eq:conv}
\bigl\|y_t-\exp\hat\psi_t\circ y_0\bigr\|_{L^2}
\leq C\bigl(m,\|y_0\|_2\bigr)\,t^{m+1/2}\,.
\end{equation}
\end{theorem}

\section{Geometric Castell--Gaines methods}
Consider the truncated exponential Lie series $\hat\psi_{t_n,t_{n+1}}$ 
across the interval $[t_n,t_{n+1}]$. 
We approximate higher order multiple Stratonovich integrals 
across each time-step by their
expectations conditioned on the increments 
of the Wiener processes on suitable subdivisions 
(Gaines and Lyons~\cite{GL}). An approximation
to the solution of the stochastic differential equation~\eqref{sde}
across the interval $[t_n,t_{n+1}]$ is given by the 
flow generated by the truncated and conditioned 
exponential Lie series $\hat\psi_{t_n,t_{n+1}}$ via
\begin{equation*}
y_{t_{n+1}}\approx\exp\bigl(\hat\psi_{t_n,t_{n+1}}\bigr)\,\circ y_{t_n}\,.
\end{equation*}
Hence the solution to the stochastic 
differential equation~\eqref{sde} can be approximately computed
by solving the ordinary differential system 
(see Castell and Gaines~\cite{CG}; Misawa~\cite{Mi})
\begin{equation}\label{castellgaines}
u'(\tau)=\hat\psi_{t_n,t_{n+1}}\circ u(\tau)
\end{equation}
across the interval $\tau\in[0,1]$. Then if
$u(0)=y_{t_n}$ we will get $u(1)\approx y_{t_{n+1}}$. 
We must choose a sufficiently accurate ordinary differential integrator 
to solve~\eqref{castellgaines}---we implicitly assume this henceforth.

The set of governing vector fields $V_i$, $i=0,1,\ldots,d$,
prescribes a map from the driving path process $w\equiv(W^1,\ldots,W^d)$ 
to the unique solution process $y\in\M$ 
to the stochastic differential equation~\eqref{sde}.
The map $w\mapsto y$ is called the It\^o map. 
Recall that we assume the vector fields are smooth.
When there is only one driving Wiener process ($d=1$) 
the It\^o map is continuous in the topology of uniform convergence
(Theorem~1.1.1.\ in Lyons and Qian~\cite{LQ}).
When there are two or more driving processes ($d\geq2$)
the Universal Limit Theorem (Theorem~6.2.2.\ 
in Lyons and Qian~\cite{LQ}) tells us that
the It\^o map is continuous in the 
$p$-variation topology, in particular for 
$2\leq p<3$. A Wiener path with $d\geq2$ 
has $p$-variation with $p>2$, and the $p$-variation 
metric in this case includes information
about the L\'evy chordal areas of the path 
(Lyons~\cite{L}). Hence we must 
choose suitable piecewise smooth approximations
to the driving path process $w$.
The following result follows from the corresponding result for 
ordinary differential equations in  
Hairer, Lubich and Wanner~\cite{HLW} (p.~112) as well
as directly from Chapter VIII in Malliavin~\cite{Malliavin}
on the Transfer Principle (see also Emery~\cite{Epaper}).

\begin{lemma}\label{tngt}
A necessary and sufficient condition for
the solution to the stochastic differential equation~\eqref{sde}
to evolve on a smooth $n$-dimensional submanifold $\mathcal M$ of\/ $\R^N$ 
($n\leq N$) up to a stopping time $T_*$
is that $V_i(y,t)\in T_y\mathcal M$ for all $y\in\mathcal M$,
$i=0,1,\ldots,d$.
\end{lemma}

Hence the stochastic Taylor expansion for the flow $\varphi_t$
is a diffeomorphism on $\M$. However a truncated version
of the stochastic Taylor expansion for the flow $\hat\varphi_t$ 
will not in general keep you on the manifold, i.e.\ if $y_0\in\M$
then $\hat\varphi_t\circ y_0$ need not necessarily lie in $\M$.
On the other hand, the exponential Lie series $\psi_t$, or any truncation 
$\hat\psi_t$ of it, lies in $\mathfrak X(\mathcal M)$. 
By Lemma~\ref{tngt} this is a necessary and sufficient condition 
for the corresponding flow-map $\exp\hat\psi_t$ to
be a diffeomorphism on $\mathcal M$.
Hence if $u(0)=y_{t_n}\in\mathcal M$, 
then $y_{t_{n+1}}\approx u(1)\in\mathcal M$.
When solving the ordinary differential equation~\eqref{castellgaines},
classical geometric integration methods, 
for example Lie group integrators such as 
Runge--Kutta Munthe-Kaas
methods, over the interval $\tau\in[0,1]$
will numerically ensure 
$y_{t_{n+1}}$ stays in $\mathcal M$. 
Additionally, as the following result reveals, numerical
methods constructed using the Castell--Gaines Lie series
approach can also be more accurate 
(a proof is provided in Appendix~\ref{secondbigproof}).
We define the \emph{strong global error} at time $T$ associated with
an approximate solution $\hat y_T$ as
$\mathcal E\equiv\|y_T-\hat y_T\|_{L^2}$.
%\smallskip

\begin{theorem}\label{th:uls}
In the case of two independent Wiener processes
and under the assumptions of Theorem~\ref{th:conv},
for any initial condition $y_0\in\mathcal M$ and a
sufficiently small fixed stepsize $h=t_{n+1}-t_n$,
the order~$1/2$ Lie series integrator
is globally more accurate in $L^2$
than the order~$1/2$ stochastic Taylor integrator.
In addition, in the case of one Wiener process, 
the order~$1$ and $3/2$ 
\emph{uniformly accurate exponential Lie series integrators}
generated by $\hat\psi^{(1)}_{t_n,t_{n+1}}=J_0V_0
+J_1V_1+\tfrac{h^2}{12}\bigl([V_1,[V_1,V_0]]\bigr)$ and
\begin{equation*}
\hat\psi^{(3/2)}_{t_n,t_{n+1}}=J_0V_0
+J_1V_1+\tfrac12(J_{01}-J_{10})[V_0,V_1]
+\tfrac{h^2}{12}\bigl([V_1,[V_1,V_0]]\bigr)\,,
\end{equation*}
respectively, are globally more accurate in $L^2$ than their
corresponding stochastic Taylor integrators. In other words,
if $\mathcal E^{\mathrm{ls}}_m$ denotes the global error of 
the exponential Lie series integrators of order $m$ above, 
and $\mathcal E^{\mathrm{st}}_m$ is the global error 
of the stochastic Taylor integrators of the 
corresponding order, then 
$\mathcal E^{\mathrm{ls}}_m\leq\mathcal E^{\mathrm{st}}_m$ for
$m=1/2,1,3/2$.
\end{theorem}
%\smallskip

\emph{Remarks.}
\begin{remunerate}
\item The result for $\hat\psi^{(3/2)}$ is new.
That the order-$1/2$ Lie series integrator
(for two Wiener processes) and the 
order~$1$ integrator generated by $\hat\psi^{(1)}$ are
uniformly more accurate confirms the asymptotically efficient
properties of these schemes proved by Castell and Gaines~\cite{CG}.
The proof follows along the lines of an analogous result
for linear stochastic systems considered in 
Lord, Malham and Wiese~\cite{LMW}.
\item Consider the order~$1/2$ exponential Lie series with
no vector field commutations. Solving the 
ordinary differential equation~\eqref{castellgaines}
using an (ordinary) Euler Munthe-Kaas method and approximating 
$\mathrm{dexp}^{-1}_\sigma\approx\I$ is equivalent to the 
order~$1/2$ stochastic Taylor Munthe-Kaas method
(for the same Lie group and action).
\end{remunerate}

\section{Numerical examples}

\subsection{Rigid body} We consider the dynamics of a rigid body
such as a satellite (see Marsden and Ratiu~\cite{MR}).
We will suppose that the rigid body is perturbed by two
independent multiplicative stochastic processes $W^1$
and $W^2$ with the corresponding vector fields
$V_i(y)\equiv\xi_i(y)\,y$,
for $i=0,1,2$, with $\xi_i\in\mathfrak{so}(3)$.
If we normalize the initial data $y_0$ so that $|y_0|=1$
then the dynamics evolves on $\M=\mathbb S^2$. We
naturally suppose $\G=SO(3)$, 
and $\Lambda_{y_0}(S)\equiv Sy_0$ so that 
$\lambda_{\xi_i}(y)=\xi_i(y)\,y$, 
and we can pull back the flow generated by $V$ on $\M$ 
to the flow on $\G$ generated by
$X_{\xi_i}(S,t)=\xi_i\bigl(\Lambda_{y_0}(S)\bigr)\,S$, $i=0,1,2$. 
We use the following matrix representation 
for the $\xi_i(y)\in\mathfrak{so}(3)$:
\begin{equation*}
\xi_i(y)=\begin{pmatrix}
0 & -y_3/\alpha_{i,3} & y_2/\alpha_{i,2} \\
y_3/\alpha_{i,3} & 0 & -y_1/\alpha_{i,1} \\
-y_2/\alpha_{i,2} & y_1/\alpha_{i,1} & 0 
\end{pmatrix}\,,
\end{equation*}
where the constants $\alpha_{i,j}$ for $j=1,2,3$ are chosen
so that the vector fields $V_i$ and matrices $\xi_i$ do not
commute for $i=0,1,2$: $\alpha_{0,1}=3$, $\alpha_{0,2}=1$, $\alpha_{0,3}=2$,
$\alpha_{1,1}=1$, $\alpha_{1,2}=1/2$, $\alpha_{1,3}=3/2$,
$\alpha_{2,1}=1/4$, $\alpha_{2,2}=1$, $\alpha_{2,3}=1/2$.
The vector fields $V_i$ 
satisfy the conditions of 
Theorem~\ref{th:conv} since the manifold is compact in this case.

\begin{figure}
  \begin{center}
  \includegraphics[width=0.9\textwidth]{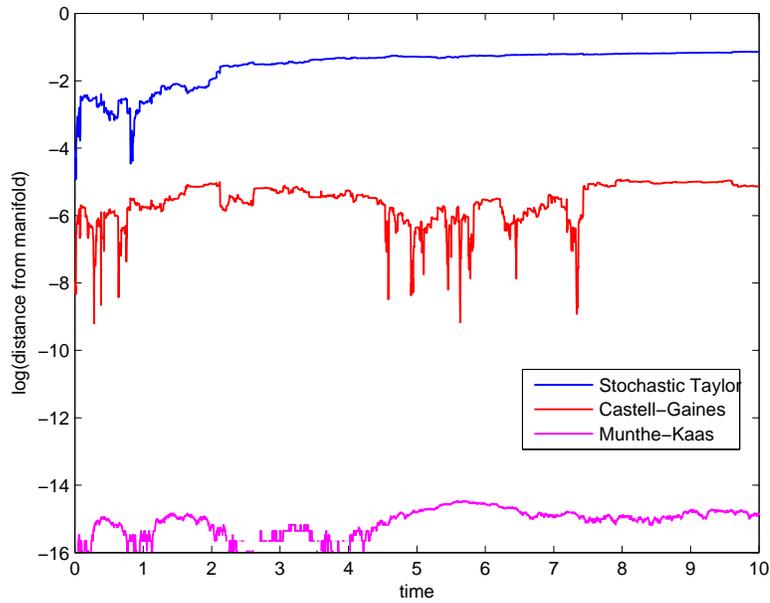}
\includegraphics[width=0.9\textwidth]{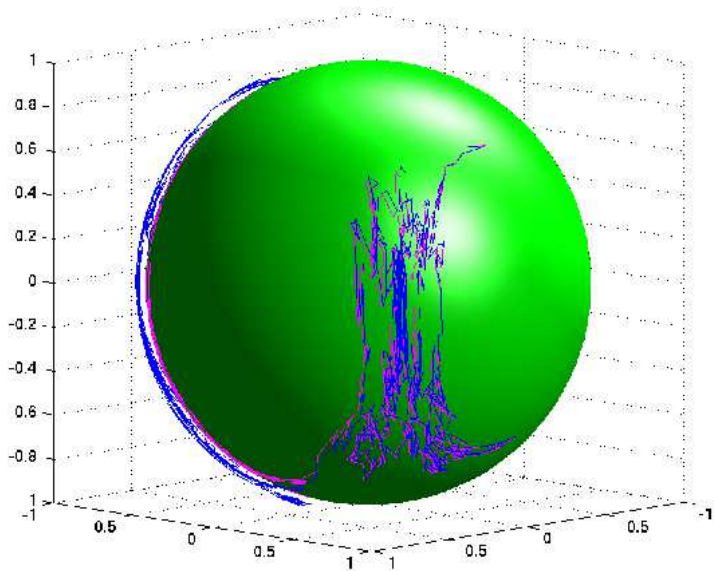}
  \end{center}
  \caption{\emph{Rigid body:} 
We show the log-distance of the approximate solution 
to the unit sphere as a function of time for each of the methods. 
Below we show the approximate solutions as a function of time
for the stochastic Taylor (blue) and Munthe-Kaas methods (magenta).
The trajectory starts at the top right and eventually drifts over
the left horizon.}
  \label{errorfig}
\end{figure}

We will numerically solve~\eqref{sde} using three
different order~$1$ methods: stochastic Taylor,
stochastic Taylor Munthe-Kaas based on~\eqref{STMK}
and Castell--Gaines (a standard non-geometric
Runge--Kutta method is used to solve 
the ordinary differential equation~\eqref{castellgaines}).
The vector field compositions $V_iV_j$ needed for the 
stochastic Taylor and Castell--Gaines methods are readily computed.
For the Munthe-Kaas method we note that
we have $v_{\xi_i}\circ\oo=\xi_i(y_0)$ and 
\begin{equation*}
v_{\xi_i}v_{\xi_j}\circ\oo=\hat A(y_0,y_0;\alpha_i,\alpha_j)
-\tfrac12[\xi_i(y_0),\xi_j(y_0)]\,.
\end{equation*}
Here $\oo\in\mathfrak{so}(3)$ is the zero element on the Lie algebra,
and for all $y,z\in\R^3$ we define
\begin{equation*}
A(y,z;\alpha,\beta)\equiv
\begin{pmatrix} 
\bigl(\tfrac{y_2z_3}{\alpha_{2}}-\tfrac{y_3z_2}{\alpha_{3}}\bigr)\tfrac{1}{\beta_{1}}\\
\bigl(\tfrac{y_3z_1}{\alpha_{3}}-\tfrac{y_1z_3}{\alpha_{1}}\bigr)\tfrac{1}{\beta_{2}}\\
\bigl(\tfrac{y_1z_2}{\alpha_{1}}-\tfrac{y_2z_1}{\alpha_{2}}\bigr)\tfrac{1}{\beta_{3}}
\end{pmatrix}\,,
\end{equation*}
and $\hat{\phantom{u}}\colon\R^3\ra\mathfrak{so}(3)$ denotes
the vector space isomorphism $\sigma\mapsto\hat\sigma$ where
\begin{equation*}
\hat\sigma\equiv\begin{pmatrix} 
0 & -\sigma_3 & \sigma_2 \\
\sigma_3 & 0 & -\sigma_1 \\
-\sigma_2 & \sigma_1 & 0
\end{pmatrix}\,.
\end{equation*}
Note that $\hat y\,z\equiv y\wedge z$ 
(see Marsden and Ratiu~\cite{MR}).
Note also since $\sigma\in\mathfrak{so}(3)$, 
$\exp\sigma\in SO(3)$ can be conveniently and cheaply computed
using Rodrigues' formula (see Marsden and Ratiu~\cite{MR}
or Iserles \textit{et al.\ }~\cite{IMNZ}).

In Figure~\ref{errorfig} we show the distance from the 
manifold $\mathbb S^2$ of each the three approximations; we
start with initial data $y_0=(\sqrt 2,\sqrt 2, 0)^{\text{T}}$.
The stochastic Taylor Munthe-Kaas method can be seen to 
preserve the solution in the unit sphere to within machine error.
We also see that the stochastic Taylor method
clearly drifts off the sphere as the integration time progresses,
as does the non-geometric Castell-Gaines method---which does however
remain markedly closer to the manifold than the stochastic Taylor scheme.

\begin{figure}[ht]
  \begin{center}
  \includegraphics[width=0.86\textwidth]{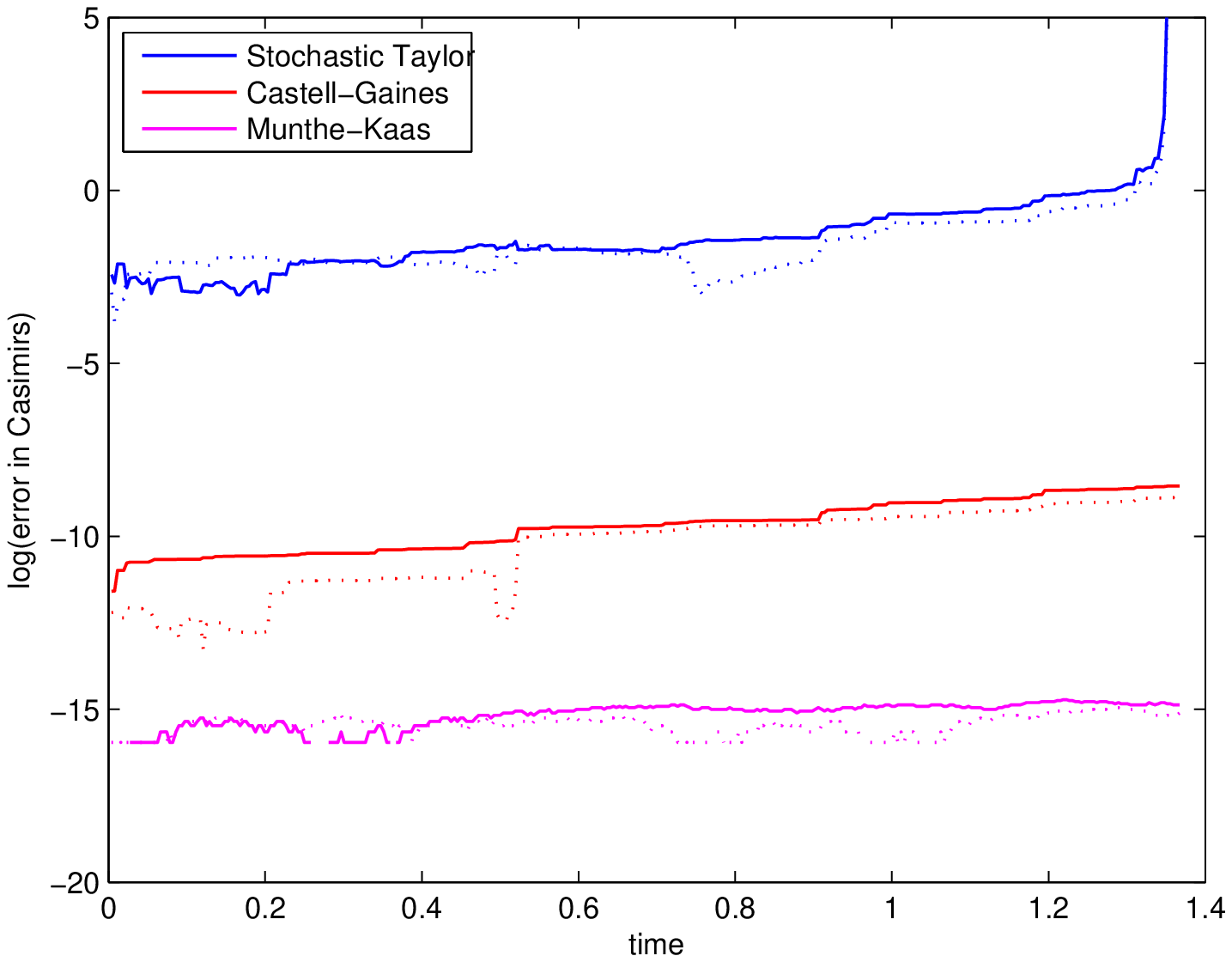}
\includegraphics[width=0.86\textwidth]{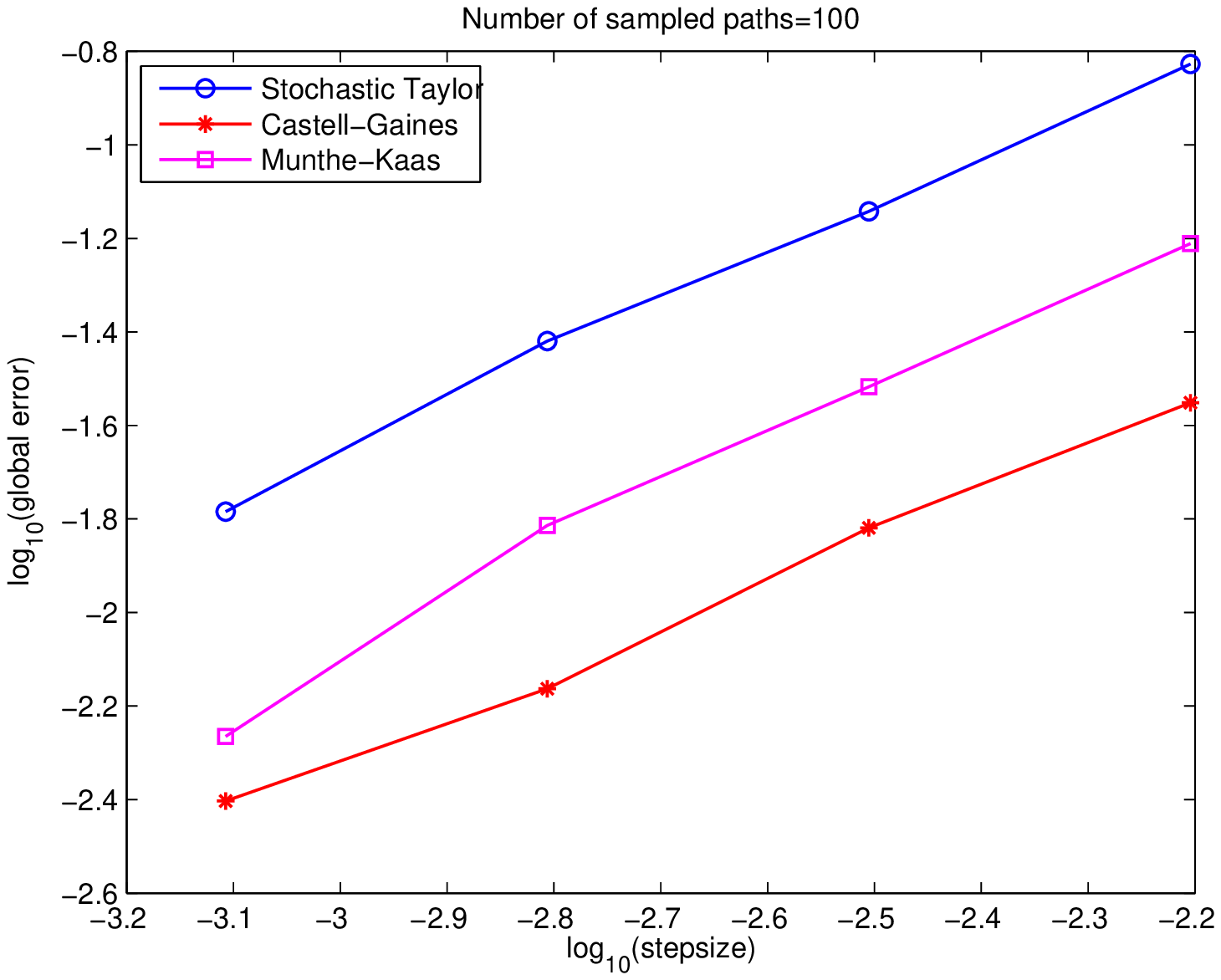}
  \end{center}
  \caption{\emph{Autonomous underwater vehicle:} 
We show the log-distance of the approximate solution 
to the two Casimirs $C_1=\pi\cdot p$ (dotted line)
and $C_2=|p|^2$ (solid line) as a function of time 
for each of the methods. Below, we also show the 
global error as a function of stepsize.}
  \label{errorfig2}
\end{figure}

\subsection{Autonomous underwater vehicle}
The dynamics of an ellipsoidal autonomous underwater vehicle 
is prescribed by the state $y=(\pi,p)\in\mathfrak{se}(3)^\ast$ 
where $\pi\in\mathfrak{so}(3)^\ast$ is its angular momentum and 
$p\in(\R^3)^\ast$ its linear momentum 
(see Holmes, Jenkins and Leonard~\cite{HJL},
Egeland, Dalsmo and S\o rdalen~\cite{EDS}
and Marsden and Ratiu~\cite{MR}).
We suppose that the vehicle is
perturbed by two independent multiplicative stochastic
processes. The governing vector
fields are for $i=0,1,2$:
\begin{equation*}
V_i(y)=\mathrm{ad}^\ast_{\xi_i}\circ y\,.
\end{equation*}
Here $\xi_i(y)=\bigl(\omega_i(y),u_i(y)\bigr)\in\mathfrak{se}(3)$
where $\omega_i(y)=I_i^{-1}\pi$ and $u_i(y)=M_i^{-1}p$
are the angular and linear velocity, and 
$I_i=\mathrm{diag}(\alpha_{i,1},\alpha_{i,2},\alpha_{i,3})$ and 
$M_i=\mathrm{diag}(\beta_{i,1},\beta_{i,2},\beta_{i,3})$ are the constant
moment of inertia and mass matrices, respectively. Explicitly for 
$\xi\in\mathfrak{se}(3)$ we have
\begin{equation*}
\mathrm{ad}^\ast_{\xi}\circ y\equiv(\pi\wedge\omega+p\wedge u,p\wedge\omega)\,.
\end{equation*}
The system of vector fields $V_i$, $i=0,1,2$ represents
the Lie--Poisson dynamics on $\M=\mathfrak{se}(3)^\ast$
(Marsden and Ratiu~\cite{MR}). There are two independent
Casimir functions 
$C_k\colon\mathfrak{se}(3)^\ast\ra\R$,
$k=1,2$, namely $C_1=\pi\cdot p$ and $C_2=|p|^2$;
these are conserved by the flow 
on $\mathfrak{se}(3)^\ast$. Note that the 
Hamiltonian, i.e.\ total kinetic energy 
$\tfrac12(\pi\cdot\omega+p\cdot u)$, is 
also exactly conserved (and helpful for establishing
the sufficiency conditions in Theorem~\ref{th:conv}), but that
is not our focus here.

If $\G=SE(3)\cong SO(3)\times\R^3$,
then the coadjoint action of $SE(3)$ on $\mathfrak{se}(3)^\ast$, 
$\mathrm{Ad}^\ast\colon SE(3)\times\mathfrak{se}(3)^\ast\ra\mathfrak{se}(3)^\ast$
is defined for all $S=(s,\rho)\in SE(3)$, where $s\in SO(3)$ and
$\rho\in\R^3$, and $y\in\mathfrak{se}(3)^\ast$ by:
$\Lambda_y\circ S=\mathrm{Ad}^\ast_{S^{-1}}\circ y
\equiv\bigl(s\pi+\rho\wedge(sp),sp\bigr)$.
The corresponding infinitesimal action
$\lambda\colon\mathfrak{se}(3)\times\mathfrak{se}(3)^\ast\ra\mathfrak{se}(3)^\ast$
for all $\xi\in\mathfrak{se}(3)$ and $y\in\mathfrak{se}(3)^\ast$
is given by (see Marsden and Ratiu~\cite{MR}, p.~477)
\begin{equation*}
\lambda_\xi\circ y=-\mathrm{ad}^\ast_{\xi}\circ y\,.
\end{equation*}
Since $\mathrm{ad}^\ast_{\xi}(y)=-\lambda_\xi(y)=\lambda_{-\xi}(y)$
the governing set of vector fields on $\mathfrak{se}(3)^\ast$ are
\begin{equation*}
V_i(y)=\lambda_{-\xi_i}\circ y\,.
\end{equation*}
We can now pull back this flow on $\mathfrak{se}(3)^\ast$
to a flow on $SE(3)$ via $\Lambda_{y_0}$. The corresponding
flow on $SE(3)$ is generated by the governing 
set of vector fields for $i=0,1,2$:
\begin{equation*}
X_{-\xi_i}\circ S=-\bigl(\omega_i(y)\wedge s,\omega_i(y)\wedge\rho+u_i(y)\bigr)\,,
\end{equation*}
with $y=\Lambda_{y_0}(S)$.

To aid implementation note that 
$SE(3)=\bigl\{(s,\rho)\in SE(3)\colon s\in SO(3),~\rho\in\R^3\bigr\}$
embeds into $SL(4;\R)$ via the map
\begin{equation*}
S=(s,\rho)\mapsto\begin{pmatrix} s & \rho \\ O^{\text{T}} & 1 \end{pmatrix}\,,
\end{equation*}
where $O$ is the three-vector of zeros. Also
$\mathfrak{se}(3)$ is isomorphic to a Lie subalgebra of
$\mathfrak{sl}(4;\R)$ with elements of the form
\begin{equation*}
\sigma=(\theta,\zeta)\mapsto
\begin{pmatrix} \hat\theta & \zeta \\ O^{\text{T}} & 0 \end{pmatrix}\,.
\end{equation*}
Hence the governing vector fields on $SE(3)$ are of the form
$X_{\xi_i}=-\xi_i(y)\,S$, where 
\begin{equation*}
\xi_i(y)=\begin{pmatrix} \hat\omega_i(\pi) & u_i(p) \\ 
                         O^{\text{T}} & 0 \end{pmatrix}\,.
\end{equation*}
The governing vector fields on $\mathfrak{se}(3)$ are 
$v_{\xi_i}(\sigma)=-\mathrm{dexp}_\sigma\circ\xi_i\bigl(\Lambda_{y_0}(\exp\sigma)\bigr)$.
Again the vector field compositions $V_iV_j$ needed for the 
stochastic Taylor and Castell--Gaines methods can be computed
straightforwardly. Direct calculation also 
reveals that in block matrix form
\begin{equation*}
v_{\xi_i}v_{\xi_j}\circ\oo=\begin{pmatrix}
\hat A(\pi_0,\pi_0;\alpha_i,\alpha_j)+\hat A(p_0,p_0;\beta_i,\alpha_j) &
A(\pi_0,p_0;\alpha_i,\beta_j) \\
O^{\text{T}} & 0\end{pmatrix}-\tfrac12\,[\xi_i(y_0),\xi_j(y_0)]\,.
\end{equation*}
Here $A(y,z;\alpha,\beta)$ is defined as for the rigid body example.
Note that the exponential map 
$\exp_{\mathfrak{se}(3)}\colon\mathfrak{se}(3)\ra SE(3)$ is
defined for all $\sigma=(\theta,\zeta)\in\mathfrak{se}(3)$ by
\begin{equation*}
\exp_{\mathfrak{se}(3)}\sigma=\begin{pmatrix}
\exp_{\mathfrak{so}(3)}\hat\theta & f(\theta)\zeta \\
O^{\text{T}} & 1\end{pmatrix}\,,
\end{equation*}
where $\exp_{\mathfrak{so}(3)}$ is the exponential map
from $\mathfrak{so}(3)$ to $SO(3)$ which can be
computed using Rodrigues' formula and 
(see Bullo and Murray~\cite{BM}, p.~5)
\begin{equation*}
f(\theta)=I_{3\times 3}+(1-\cos\|\theta\|)\hat\theta/\|\theta\|^2
+\bigl(1-(\sin\|\theta\|)/\|\theta\|\bigr)\hat\theta^2/\|\theta\|^2\,.
\end{equation*}

In Figure~\ref{errorfig2} we show the distance from the 
manifold $\mathfrak{se}(3)^\ast$ of each the three approximations; 
in particular how far
the individual trajectories stray from the Casimirs
$C_1=\pi\cdot p$ and $C_2=|p|^2$. We
start with the initial data 
$y_0=(\sqrt 2,\sqrt 2, 0, 0, \sqrt 2,\sqrt 2)^{\text{T}}$.
As before the stochastic Taylor Munthe-Kaas method can be seen to 
preserve the Casimirs to within machine error.
We also see that the stochastic Taylor method
clearly drifts off the manifold as the integration time progresses
and at a particular time depending on the Wiener path
shoots off very rapidly away from the manifold. 
Note also that for large stepsizes the stochastic Taylor method
is unstable. However
the non-geometric Castell--Gaines and stochastic Munthe-Kaas methods
still give reliable results in that regime. Lastly, although the 
the stochastic Munthe-Kaas method adheres to the manifold to
within machine error, the error of the non-geometric 
Castell--Gaines method is actually smaller.

\section{Conclusions}
We have established and implemented stochastic Lie group
integrators based on stochastic Munthe-Kaas methods and also
derived geometric Castell--Gaines methods. We have also
revealed several aspects of these integrators 
that require further investigation.
\begin{remunerate}
\item We could construct a stochastic nonlinear Magnus method 
by approximating the solution to the stochastic differential
equation~\eqref{sigmaeqn} on the Lie algebra using  
Picard iterations (see Casas and Iserles~\cite{CI}).
\item We would like to develop a practical procedure 
for implementing ordinary Munthe-Kaas methods for 
higher order Castell--Gaines integrators. We need to determine
the element $\xi\colon\M\ra\g$ so that in~\eqref{castellgaines}
we have $\hat\psi=\lambda_\xi$. 
\item We need to determine the properties of
the local and global errors for the stochastic
Munthe-Kaas methods. Also a thorough investigation 
of the stability properties of
the stochastic Munthe-Kaas and Castell--Gaines 
methods is required. For the autonomous underwater vehicle 
simulations they were both superior to the direct stochastic
Taylor method, especially for larger stepsizes.
We also need to compare the relative efficiency of
the methods concerned, in particular to compare 
an optimally efficient geometric Castell--Gaines method
with the stochastic Munthe-Kaas method. 
\item Although we have chiefly confined ourselves to 
driving paths that are Wiener processes, we can extend 
Munthe-Kaas and Castell--Gaines methods to
rougher driving paths (Lyons and Qian~\cite{LQ},
Friz~\cite{F}, Friz and Victoir~\cite{FV}). 
Further, what happens when we consider processes involving jumps?
For example Srivastava, Miller and Grenander~\cite{SMG} consider
jump diffusion processes on matrix Lie groups for Bayesian
inference. Or what if we consider fractional Brownian 
driving paths; Baudoin and Coutin~\cite{BC} 
investigate fractional Brownian motions on Lie groups?
\item Schiff and Shnider~\cite{SS} have used Lie group methods to
derive M\"obius schemes for numerically integrating 
deterministic Riccati systems beyond finite time removable singularities 
and numerical instabilities. They integrate a linear
system of equations on the general linear group $\mathrm{GL}(n)$ which
corresponds to a Riccati flow on the Grassmannian manifold $\mathrm{Gr}(k,n)$
via the M\"obius action map. Lord, Malham and Wiese~\cite{LMW}
implemented stochastic M\"obius schemes and show that they
can be more accurate and cost effective than directly solving
stochastic Riccati systems using stochastic Taylor methods.
We would like to investigate further their effectiveness for 
stochastic Riccati equations arising
in Kalman filtering (Kloeden and Platen~\cite{KP}) 
and to backward stochastic Riccati equations arising in
optimal stochastic linear-quadratic control 
(see for example Kohlmann and Tang~\cite{KT} and
Estrade and Pontier~\cite{EP}).
\item Other areas of potential application of the 
methods we have presented in this paper are for example:
term-structure interest rate models evolving on
finite dimensional invariant manifolds (see
Filipovic and Teichmann~\cite{FT}); stochastic dynamics triggered 
by DNA damage (Chickarmane, Ray, Sauro and Nadim~\cite{CRSN})
and stochastic symplectic integrators
for which the gradient of the solution evolves on the symplectic
Lie group (see Milstein and Tretyakov~\cite{MT}).
\end{remunerate}

\section*{Acknowledgments}
We thank Alex Dragt, Peter Friz, Anders Hansen, Terry Lyons, 
Per-Christian Moan and Hans Munthe--Kaas for stimulating discussions. 
We also thank the anonymous referees, whose suggestions 
and encouragement improved the original manuscript significantly.
SJAM would like to acknowledge the invaluable facilities
of the Isaac Newton Institute where some of the final
touches to this manuscript were completed.

\appendix

%\section{Proof that $X_\xi$ is not in general right invariant} 
%Suppose $R_Q\colon\G\ra\G$ is the right translation map on $\G$, 
%i.e.\ $R_Q(S)=SQ$. Then if $\alpha(\tau)\in\G$ is such that
%$\alpha(0)=S$ and 
%$\partial_\tau\alpha(\tau)=X_\xi\bigl(\alpha(\tau)\bigr)$
%we have that the push forward of $X_\xi$ across $\G$ by $R_Q$ is 
%\begin{align*}
%\bigl((R_Q)_\ast X_\xi\bigr)(SQ)=&\;\left.\partial_\tau 
%R_Q\circ\alpha(\tau)\right|_{\tau=0}\\
%=&\;\left.\partial_\tau \alpha(\tau)Q\right|_{\tau=0}\\
%=&\;\xi(t,S)\,SQ\\
%\neq&\;X_\xi(SQ)\,.
%\end{align*}

%\section{Proof of the push forward formula for 
%$\bigl(\Lambda(S)\bigr)_\ast\lambda_\xi$}
%If $y=\Lambda(S)\circ y_0$, the push forward
%of $\lambda_\xi$ from $y_0\in\M$ to $y\in\M$ by $\Lambda(S)$
%is given as follows. Suppose $\gamma(t)\in\G$, $\gamma(0)=\I$,
%$\partial_\tau\gamma(\tau)=X_\xi\circ\gamma(\tau)$. Then
%if $\alpha(\tau)=\Lambda(S)\circ y_0$, we have
%\begin{align*}
%\bigl(\Lambda(S)\bigr)_\ast\lambda_\xi\circ y
%=&\;\left.\partial_\tau\Lambda(S)\circ \alpha(\tau)\right|_{\tau=0}\\
%=&\;\left.\partial_\tau\Lambda(S)\circ
%\Lambda\bigl(\gamma(\tau)\bigr)\circ y_0\right|_{\tau=0}\\
%=&\;\left.\partial_\tau\Lambda(S)\circ\Lambda\bigl(\gamma(\tau)\bigr)
%\circ\Lambda\bigl(S^{-1}\bigr) \circ y\right|_{\tau=0}\\
%=&\;\left.\partial_\tau\Lambda\bigl(S\gamma(\tau)S^{-1}\bigr) 
%\circ y\right|_{\tau=0}\\
%=&\;\lambda_{\Ad_S\xi}\circ y\,.
%\end{align*}

\section{Proof of Theorem~\ref{th:conv}}\label{firstbigproof}
We follow the proof for 
linear stochastic differential equations in 
Lord, Malham and Wiese~\cite{LMW} (where further
technical details on estimates for 
multiple Stratonovich integrals can be found).
Suppose $\hat\psi_t\equiv\hat\psi_t(m)$
is the truncated Lie series~\eqref{eq:orderm}.
First we show that $\exp\hat\psi_t\circ y_0\in L^2$. 
We see that for any number $k$, $\bigl(\hat\psi_t\bigr)^k\circ y_0$ 
is a sum of $|\mathbb Q_m|^k$ terms,
each of which is a $k$-multiple product of terms 
$J_\alpha\, c_\alpha\circ y_0$.
It follows that
\begin{equation} \label{eq:product}
\bigl\|\bigl(\hat\psi_t\bigr)^k\circ y_0\bigr\|_{L^2}\leq 
\Bigl(\,\max_{\alpha\in\mathbb Q_m}\|c_\alpha\circ y_0\|\Bigr)^k\,\cdot
\sum_{\stackrel{\alpha_i \in\mathbb Q_m}{i=1,\ldots,k}}
\|J_{\alpha_1}J_{\alpha_2}\cdots J_{\alpha_k}\|_{L^{2}}\,.
\end{equation}
Note that the maximum of the norm of the 
compositions of vector fields $c_\alpha\circ y_0$
is taken over a finite set.
Repeated application of the product rule reveals that
for $i=1,\ldots,k$, each term `$J_{\alpha_1}J_{\alpha_2}\cdots J_{\alpha_k}$' 
in \eqref{eq:product} is the sum of at most 
$2^{2mk-1}$ Stratonovich integrals $J_\beta$, where for $t\leq1$, 
$\|J_{\beta}\|_{L^{2}}\leq2^{4mk-1}\,t^{k/2}$.
Since the right hand side of equation \eqref{eq:product}
consists of $|\mathbb Q_m|^k\,2^{2mk-1}$ Stratonovich integrals $J_\beta$,
we conclude that,
\begin{equation*}
\Bigl\|\bigl(\hat\psi_t\bigr)^k\circ y_0\Bigr\|_{L^2}
\leq\Bigl(\,\max_{\alpha\in\mathbb Q_m}\|c_\alpha\circ y_0\|
\cdot|\mathbb Q_m|\cdot2^{6m}\cdot t^{1/2}\Bigr)^k\,.
\end{equation*}
Hence $\exp\hat\psi_t\circ y_0$ is square-integrable.

Second we prove~\eqref{eq:conv}. Let $\hat y_t$ denote 
the stochastic Taylor series solution, truncated 
to included terms of order up to and including $t^m$.
We have
\begin{equation*}
\bigl\|y_t-\exp\hat\psi_t\circ y_0\bigr\|_{L^2}\leq
\bigl\|y_t-\hat y_t\bigr\|_{L^2}
+\bigl\|\hat y_t-\exp\hat\psi_t\circ y_0\bigr\|_{L^2}\,.
\end{equation*}
We know $y_t\in L^2$---see Lemma III.2.1 in 
Gihman and Skorohod~\cite{GS}. Note that the assumptions
there are fulfilled, since the uniform boundedness of the
derivatives implies uniform Lipschitz continuity of the 
vector fields by the mean value theorem, and uniform 
Lipschitz continuity in turn implies a linear growth 
condition for the vector fields since they are autonomous.
Note that $\hat y_t$ is a strong approximation to $y_t$ 
up to and including terms of order $t^m$,
with the remainder consisting of $\mathcal O(t^{m+1/2})$ terms
(see Proposition 5.9.1 in Kloeden and Platen~\cite{KP}).
It follows from the definition of the exponential Lie series as
the logarithm of the stochastic Taylor series, 
that the terms of order up to and including $t^m$
in $\exp\hat\psi_t\circ y_0$ correspond with $\hat y_t$; 
the error consists of $\mathcal O(t^{m+1/2})$ terms.

\section{Proof of Theorem~\ref{th:uls}}\label{secondbigproof}
Our proof follows along the lines of that for
uniformly accurate Magnus integrators for 
linear constant coefficient systems 
(see Lord, Malham \& Wiese~\cite{LMW} and Malham and Wiese~\cite{MW}).
Let $\varphi_{t_n,t_{n+1}}$ and $\hat\varphi_{t_n,t_{n+1}}$
denote the exact and approximate flow-maps 
constructed on the interval $[t_n,t_{n+1}]$ 
of length $h$.
We define the local flow remainder as 
\begin{equation*}
R_{t_n,t_{n+1}}\equiv
\varphi_{t_n,t_{n+1}}-\hat\varphi_{t_n,t_{n+1}}\,,
\end{equation*}
and so the local remainder is $R_{t_n,t_{n+1}}\circ y_{t_n}$.
Let $R^{\text{ls}}$ and $R^{\text{st}}$ denote the 
local flow remainders corresponding to the exponential
Lie series and stochastic Taylor approximations, 
respectively. 

\subsection{Order $1/2$ integrator: two Wiener processes}
For the global order~$1/2$
integrators we have to leading order
$R^{\text{ls}}=\tfrac12(J_{12}-J_{21})[V_1,V_2]$
and $R^{\text{st}}=J_{12}V_1V_2+J_{21}V_2V_1$.
Note that we have included the terms 
$J_{11}V_1^2$ and $J_{22}V_2^2$ in the integrators.
A direct calculation reveals that
\begin{equation}\label{eq:compare}
\mathbb E\bigl((R^{\text{st}}\circ y_0)^{\text{T}}R^{\text{st}}\circ y_0\bigr)=
\mathbb E\bigl((R^{\text{ls}}\circ y_0)^{\text{T}}R^{\text{ls}}\circ y_0\bigr)
+h^{2m}U^{\text{T}}BU
+\mathcal O\bigl(h^{2m+\frac{1}{2}}\bigr)\,.
\end{equation}
Here $m=1/2$ (for the order~$1/2$ integrators),
$U=(V_1V_2\circ y_0,V_2V_1\circ y_0)^{\text{T}}\in\mathbb R^{2n}$,
and $B\in\mathbb R^{2n\times 2n}$ consists of 
$n\times n$ diagonal blocks of
the form $b_{ij}I_{n\times n}$ where
\begin{equation*}
b=\tfrac14\begin{pmatrix} 1 & 1 \\
                          1 & 1 
\end{pmatrix}\,,
\end{equation*}
and $I_{n\times n}$ is the $n\times n$ identity matrix.
Since $b$ is positive semi-definite,
the matrix $B=b\otimes I_{n\times n}$ is positive semi-definite.
Hence the order~$1/2$ exponential Lie series integrator is
locally more accurate than the corresponding stochastic
Taylor integrator. 

\subsection{Order $1$ integrator: one Wiener process}
For the global order~$1$ integrators we have to leading order
$R^{\text{ls}}=\tfrac12(J_{01}-J_{10})[V_0,V_1]$ and 
$R^{\text{st}}=J_{01}V_0V_1+J_{10}V_1V_0
+J_{111}V_1^3+\tfrac14h^2(V_0V_1^2+V_1^2V_0)$.
The terms of order $h^2$ shown are significant when
we consider the global error in Section~\ref{globalerror} below.
The estimate~\eqref{eq:compare} also applies in
this case with $m=1$ and 
$U=(V_0V_1\circ y_0,V_1V_0\circ y_0,V_1^3\circ y_0)^{\text{T}}\in\mathbb R^{3n}$;
and $B\in\mathbb R^{3n\times 3n}$ consists of 
$n\times n$ diagonal blocks of
the form $b_{ij}I_{n\times n}$ where
\begin{equation*}
b=\tfrac{1}{12}\begin{pmatrix} 3 & 3 & 3\\
                               3 & 3 & 3\\
                               3 & 3 & 5
\end{pmatrix}\,.
\end{equation*} 
Since $b$ is positive semi-definite, the matrix
$B=b\otimes I_{n\times n}$ is positive semi-definite.
Hence the order~$1$ exponential Lie series integrator is
locally more accurate than the corresponding stochastic
Taylor integrator.

\subsection{Order $3/2$ integrator: one Wiener process}
The local flow remainders are 
$R^{\text{ls}}=\tfrac16\bigl(J_{110}-2J_{101}+J_{011}
-\tfrac12h^2\bigr)[V_1,[V_1,V_0]]$ and
$R^{\text{st}}=J_{011}V_0V_1^2+J_{101}V_1V_0V_1+J_{110}V_1^2V_0
+J_{1111}V_1^4-\tfrac14h^2(V_0V_1^2+V_1^2V_0+\tfrac12V_1^4)$.
The terms of order $h^2$ shown are significant when
we consider the global error---but for a 
different reason this time---see Section~\ref{globalerror} below.
Again, the estimate~\eqref{eq:compare} applies in
this case with $m=3/2$ and 
$U=(V_0V_1^2\circ y_0,V_1V_0V_1\circ y_0,
V_1^2V_0\circ y_0,V_1^4\circ y_0)^{\text{T}}\in\mathbb R^{4n}$;
and $B\in\mathbb R^{4n\times 4n}$ consists of 
$n\times n$ diagonal blocks of
the form $b_{ij}I_{n\times n}$ where
\begin{equation*}
b=\tfrac{1}{144}\begin{pmatrix} 11 & 8  & 5  & 12\\
                                8  & 8  & 8  & 12\\
                                5  & 8  & 11 & 12\\
                                12 & 12 & 12 & 24
\end{pmatrix}\,.
\end{equation*}
Again, $B$ is positive semi-definite and 
the order~$3/2$ exponential Lie series integrator is
locally more accurate than the corresponding stochastic
Taylor integrator.

\subsection{Global error}\label{globalerror}
Recall that we define the \emph{strong global error} 
at time $T$ associated with
an approximate solution $\hat y_T$ as
$\mathcal E\equiv\|y_T-\hat y_T\|_{L^2}$.
The exact and approximate solutions can be constructed by successively
applying the exact and approximate flow maps $\varphi_{t_n,t_{n+1}}$
and $\hat\varphi_{t_n,t_{n+1}}$ on the successive intervals 
$[t_n,t_{n+1}]$ to the initial data $y_0$. 
A straightforward calculation shows 
for a small fixed stepsize $h$, 
\begin{equation}\label{eq:geest}
\mathcal E^2
=\mathbb E\,(\mathcal R\circ y_0)^{\text{T}}\,\mathcal R\circ y_0\,,
\end{equation}
up to higher order terms, where
$\mathcal R\equiv\sum_{n=0}^{N-1}
\varphi_{t_{n+1},t_N}\circ R_{t_n,t_{n+1}}
\circ\varphi_{t_0,t_n}$
is the standard accumulated local error contribution
to the global error. The important conclusion 
is that when we construct the 
global error~\eqref{eq:geest}, the terms of leading order
in the local flow remainders $R^{\text{ls}}$ or $R^{\text{st}}$
with zero expectation lose only a half order of convergence 
in this accumulation effect. Hence in the local flow 
remainders shown above,
for the terms of zero expectation, the local superior
accuracy for the Lie series integrators transfers to
the corresponding global errors 
(see Lord, Malham and Wiese~\cite{LMW} for more details).
Terms of non-zero expectation however 
behave like deterministic error terms losing a whole order
(in the local to global convergence); they
contribute to the global error through their expectations. 
Hence we 
include such terms of order $h^2$ in the order~$3/2$ integrators
above and they appear as the terms subtracted from the 
remainders shown. For the order~$1$ integrators we do
not need to include the order $h^2$ terms in the integrator
to obtain the correct mean-square convergence. However
to guarantee that the global error for the exponential
Lie series integrator is always smaller than that for the
stochastic Taylor scheme, we include this term in the integrator.

\label{lastpage}

\end{document}